\documentclass[12pt]{article}
\usepackage{amsthm}
\usepackage{picinpar}
\usepackage{tikz}
\usepackage{graphicx}
\usepackage{amssymb,amsmath, mathrsfs}
\newtheorem{teorema}{Theorem}

\newtheorem{proposizione}[teorema]{Proposition}

\newtheorem{definizione}[teorema]{Definition}
\theoremstyle{remark}

\title{$LS$-sequences of points in the unit square}

\author{ {\it Ingrid Carbone,  Maria Rita Iac\`{o}, Aljo\v{s}a Vol\v{c}i\v{c}}}\date{}

\begin{document}

\maketitle

\begin{abstract}
 
\noindent  We define a countable family of sequences of points in the unit square:  the {\it $LS$-sequences of points \`a la Halton}. They  reveal a very strange and interesting behaviour, as well as resonance phenomena, for which we have not found an explanation, so far. We conclude with three open problems.
\end{abstract}



\section{$LS$-sequences  in $[0,1[$}

In \cite{Carbone} the first author introduced a countable family of {\it $LS$-sequences of points} in $[0,1[$, whose construction strongly depends on the  {\it $LS$-sequences of partitions} of $[0,1[$ introduced in the same article. Actually, the former are obtained by reordering the left endpoints of the intervals of the latter. 

In this paper we make the first step in the direction of their generalization to higher dimension, presenting two possible extensions to dimension $2$: the {\it $LS$-point sets \`a la van der Corput}, and the {\it $LS$-sequences \`a la Halton}.

The interest on low discrepancy sequences of points in dimension two or higher is motivated by their application in Quasi-Monte Carlo methods (see \cite{Drmota_Tichy}).

In this section we recall the definition of the {\it $LS$-sequences of partitions} and {\it points} and their main properties. 
 
The {\it $LS$-sequences of partitions} are obtained as a particular case  of the {\it $\rho$-refinements} introduced by the third author in \cite{Volcic}.

\begin{definizione}\label{1} 
\rm{Consider any non trivial finite partition $\rho$ of $[0,1[$. The {\it $\rho$-refinement} of a partition $\pi$ of $[0,1[$ (which will be denoted by $\rho\pi$) is obtained by subdividing only the interval(s) of $\pi$ having maximal length homotetically to $\rho$. 
Denote by $\rho^{n}\pi$  the $\rho$-refinement of
 $\rho^{n-1}\pi$, and by $\{\rho^n \pi\}_{n \ge 1}$ the  sequence of successive {\it $\rho$-refinements}.
}
\end{definizione}

If $\rho=\{[0, \alpha[ ,[\alpha, 1[ \}$ and $\omega=\{[0,1[\}$  is the trivial partition of $[0,1[$, the sequence of successive $\rho$-refinements is actually the  splitting procedure introduced by Kakutani \cite{Kakutani}.

In \cite{Volcic} it has been proved that the sequence $\{\rho^n\omega\}_{n \ge 1}$ is {\it uniformly distributed}, which means that
if $\rho^n\omega=\{[y_i^{(n)}, y_{i+1}^{(n)}[\, :\, 1\leq i\leq t_n\}$, the sequence satisfies
\begin{equation*}\lim_{n\to\infty}\frac{1}{t_n} \sum_{i=1} ^{t_n} \chi_{[a, b[}(y_i^{(n)})=b-a\ ,\end{equation*}
for every pair of real numbers $a,b$, with $0\leq a < b \leq 1$ (see \cite{Volcic} for the first formal definition of this concept).

An {\it $LS$-sequence of partitions}, denoted by  $\{\rho_{L,S}^{n}\}_{n \ge 1}$,   is obtained as successive $\rho_{LS}$-refinements of $\omega$, where $\rho_{L,S}$ is the partition made by $L$ long intervals having length $\gamma$ followed by $S$ short ones having length $\gamma^{2}$, and $\gamma$ is the positive root of the equation $L\gamma + S\gamma^{2}=1$.

It is clear that each partition $\rho^n_{L,S}$ has only long intervals having length $\gamma^n$ and short ones having length $\gamma^{n+1}$.
The sequence $\{t_n\}_{n \ge 1}$ of the number of intervals of $\{\rho_{L,S}^{n}\}_{n \ge 1}$ satisfies  the difference equation $t_n = Lt_{n-1} +St_{n-2}$ with $t_0 = 1$ and $t_1 = L+S$.

In \cite{Carbone} it has been proved that when $ S \le L$ the {\it $LS$-sequence of partitions} $\{\rho^n_{L,S}\}$ has {\it low discrepancy}, namely there exists a constant $C>0$ depending on $L$ and $S$ such that $ {t_n}\,D(\rho^n_{L,S})\le C\,$ for any $n$. Here  $D(\rho^n_{L,S})$ denotes the {\it discrepancy} of $\rho^n_{L,S}$ defined by 

\begin{equation*} D(\rho^n_{L,S})=\sup_{0\leq a<b\leq 1}\left| \frac{1}{t_n}\sum_{j=1}^{t_n}\chi_{[a,b[}(y_j^{(n)})-(b-a) \right| .\end{equation*}

\bigskip
In fact, we have more generally

\begin {teorema} \label{2} 

 If  $S \le L$   there exist  $c_1, c'_1 >0$ such that   $ c'_1  \le  t_n \,D( \rho_{L,S}^n  ) \le \, c_1 $ for any $n \in \mathbb{N}$.

 If  $S =L+1$ there exist $c_2, c'_2>0$ such that $c'_2 \,{\log t_n}  \le t_n \,D( \rho_{L,S}^n  ) \le c_2 \,{\log t_n} $
for any $n \in \mathbb{N}$.

 If  $S\ge L+2$ there exist $c_3, c'_3>0$ such that $\,c'_3\, {t_n}^{1-\tau}\,  \le t_n \, D( \rho_{L,S}^n  )  \le \, c_3 \, t_n^{1-\tau}$ for any $n \in \mathbb{N}$,
where $ 1-\tau = - \frac {\log(S \gamma)} {\log \gamma}>0$.
\end{teorema}

\smallskip\smallskip
In \cite{Carbone2} the first author (improving  \cite{Carbone}) introduced an algorithm which associates to each {\it $LS$-sequence of partitions} a sequence (of points), called {\it $LS$-sequence of points} and denoted by $\{ \xi_{L,S}^n\}_{n \ge 1}$.
The sequences  $\{ \xi_{L,S}^n\}_{n \ge 1}$ can be seen in terms of the representation in base $L+S$ of a suitable subsequence of natural numbers. 
We describe briefly this construction (see \cite{Carbone2} for further details).  

Any positive integer can be written as 
\begin{equation*}n=\sum_{k=0}^{M} a_k(n)(L+S)^k ,\end{equation*}
 where $a_k(n) \in \{ 0,1,2,\dots, L+S-1 \}$ for all $0 \le k \le M$ and $M={\cal b}\log_{L+S} n \cal c$ (here ${\cal b}\cdot  \cal c$ denotes the integer part, as usual). 
 
       Let us denote by $ \mathbb{N}_{L,S} $ the infinite set of all positive integers $n$,  ordered by magnitude, such that, for each $k \in \mathbb{N}$,  $(a_{k}(n),a_{k+1}(n)) \notin   \{ L, L+1, \dots, L+S-1\} \times \{1,\dots,L+S-1 \}$. 
Define on $ \mathbb{N}_{L,S} $ the function
 
 \begin{equation*} \phi_{L,S}(n)=\sum_{k=0}^M \,  \tilde a_k(n) \, \gamma ^{k+1} \,,\end{equation*} 
where  $ \tilde a_k(n)= a_k(n) $ if $0 \le a_k(n) \le L-1 $, while 
$\tilde a_k(n) =L+\gamma(a_k(n)-L)$ if $ L\le a_k(n) \le L+S-1$.

The sequence $\{\phi_{L,S}(n)\}_{n \ge 0}$ defined on $\mathbb{N}_{L,S}$ is the {\it $LS$-sequence of points}.

The most important property these sequences show is that whenever   the {\it $LS$-sequence of points} has low discrepancy,  the corresponding {\it $LS$-sequence of points} obtained by the algorithm described above has low discrepancy too. 
More precisely, in \cite{Carbone}  upper bounds for their discrepancy have been given (see \cite{Kuipers_Niederreiter} for the general theory on uniform distribution and discrepancy).

 \smallskip 
\begin {teorema} \label{3}

 If  $S \le L$   there exists  $k_1 >0$ such that for any $N \in \mathbb{N}$ we have $  N \, D \Big(\xi_{L,S}^1, \xi_{L,S}^2, \dots, \xi_{L,S}^N  \Big) \le \,  k_1  {\log N} $. 

\smallskip
If  $S =L+1$ there exists $k_2, c'_2>0$ such that for any $N \in \mathbb{N}$ we have
$c'_2 \log N \le N \, D \Big( \xi_{L,S}^1, \xi_{L,S}^2, \dots, \xi_{L,S}^N  \Big) \le k_2 \,{\log^2 N}$.

\smallskip
 If  $S\ge L+2$ there exists $k_3, c'_3>0$ such that for any $N \in \mathbb{N}$ we have
$c'_3\, {N^{1-\tau}} \le N \, D \Big( \xi_{L,S}^1, \xi_{L,S}^2, \dots, \xi_{L,S}^N  \Big)  \le \, k_3 \, {N^{1-\tau}}\, {\log N}$,
 where $ 1-\tau = - \frac {\log(S \gamma)} {\log \gamma}>0$.
\end{teorema}


\section{$LS$-sequences of points in the unit square}

In this  section we see how the {\it $LS$-sequences of points} can be used to produce sequences in the unit square. 

One possibility is to imitate the van der Corput idea  \cite{vanderCorput}, combining an {\it $LS$-sequence of points} with the points (ordered by magnitude) associated to the Knapowski partition $\left \{\left[ \frac{i-1}{N}, \frac{i}{N}\right[  , 1 \le i \le N\right \}$ of order $N$. 
 Hammersely \cite{Hammersley} extended this definition to higher dimension.

The other possibility is to put on the two coordinates two different {\it $LS$-sequences of points}, imitating what Halton did in \cite{Halton} when he paired two van der Corput sequences having different bases. He proved that whenever these bases are coprime, the sequence has low discrepancy in the unit square.

 The first idea produces the following
\begin{definizione} 
\rm{For each $LS$-sequence of points $\{\xi_{L,S}^n\}_{n \ge 1}$, the finite set of points  
\begin{equation*}P_{L,S}(N)=\bigg \{\Big( \frac{n-1}{N}, \xi_{L,S}^n \Big),\quad n=1,\dots,N\bigg \}\end{equation*}

\smallskip\smallskip
\noindent  is called \emph{$LS$-point set \`{a} la van der Corput} of order $N$ in the unit square.}
\end{definizione}

The main result concerning these {\it $LS$-points sets \`{a} la van der Corput} is given by the following

\begin{proposizione}\label{5} 

If  $S \le L$   there exists  $C_1 >0$ such that for any $N \in \mathbb{N}$ we have $  N \, D \Big(P_{L,S}(N)  \Big) \le \,  C_1  {\log N} $. 

\smallskip
If  $S =L+1$ there exists $C_2>0$ such that for any $N \in \mathbb{N}$ we have
$N \, D \Big(P_{L,S}(N)  \Big) \le C_2 \,{\log^2 N}$.

\smallskip
 If  $S\ge L+2$ there exists $k_3, c'_3>0$ such that for any $N \in \mathbb{N}$ we have
$ N \, D \Big( P_{L,S}(N)  \Big)  \le \, C_3 \, {N^{1-\tau}}\, {\log N}$,
 where $ 1-\tau = - \frac {\log(S \gamma)} {\log \gamma}>0$.
\end{proposizione}
\begin{proof} Let us fix a rectangle $R=[0, a[\times [0, b[$ and an $LS$-sequence of points $\{ \xi_{L,S}^n\}_{n \ge 1}$. 
If we denote by $f_{L,S}$ one of the three upper bounds appearing in Theorem \ref{3}, a simple calculation gives. 

\begin{equation*}
\left| \frac{1}{N} \sum_{j=1}^{N} \chi_R\left(\frac{j-1}{N}, \xi_{L,S}^{j}\right)-ab \right|= \left| \frac{1}{N} \sum_{j=1}^{N} \chi_{[0, a[}\left(\frac{j-1}{N}\right)\chi_{[0,b[}(\xi_{L,S}^{j})-ab \right| \end{equation*} 

\begin{equation*}\leq \left| \frac{1}{N} \sum_{j=1}^{{\cal b}Na{\cal c}+1} \chi_{[0,b[}(\xi_{L,S}^{j})-\frac{{\cal b}Na{\cal c}+1}{N}b \right|+ \left| \frac{{\cal b}Na{\cal c}+1}{N}b-ab \right|\end{equation*}

\begin{equation*}=\frac{{\cal b}Na{\cal c}+1}{N} \left| \frac{1}{{\cal b}Na{\cal c}+1} \sum_{j=1}^{{\cal b}Na{\cal c}+1} \chi_{[0,b[}(\xi_{L,S}^{j})- b\right| + b\left| \frac{{\cal b}Na{\cal c}+1}{N}-a \right|\end{equation*}

\begin{equation*}\le \frac{1}{N} \left| \frac{1}{{\cal b}Na{\cal c}+1} \sum_{j=1}^{{\cal b}Na{\cal c}+1} \chi_{[0,b[}(\xi_{L,S}^{j})- b\right| + \frac{b}{N}\end{equation*}

\begin{equation*}\le \frac{1}{N} f_{L,S}({\cal b}Na{\cal c}+1) + \frac{1}{N}\le c' \frac{f_{L,S}(N)}{N}\ .\end{equation*}

\bigskip
Taking the supremum over all the rectangles $R$ in the unit square, the theorem is completely proved. 
\end{proof}

{\bf Remark} Comparing the above proposition to Theorem \ref{3}, we conclude that any {\it $LS$-point set \`{a} la van der Corput} has {\it low discrepancy} when $S \le L$.

\bigskip
Let us now give the second and more interesting generalization.

\begin{definizione}\label{6} 
\rm{Given two $LS$-sequences of points $\{\xi_{L_1,S_1}^{n}\}_{n \ge 1}$ and 
 $\{\xi_{L_2,S_2}^{n}\}_{n \ge 1}$, the sequence
  \begin{equation*}\{\xi_{L_1,S_1}^{n}, \xi_{L_2,S_2}^{n}\}_{n\geq 1}\end{equation*}

\smallskip\smallskip\smallskip
\noindent is called \emph{$LS$-sequence of points \`a la Halton} in the unit square.}
\end{definizione}

\begin{figure}[h!]
\centering
\begin{minipage}[c]{.45\textwidth}
\centering
\includegraphics[scale=0.6, keepaspectratio]{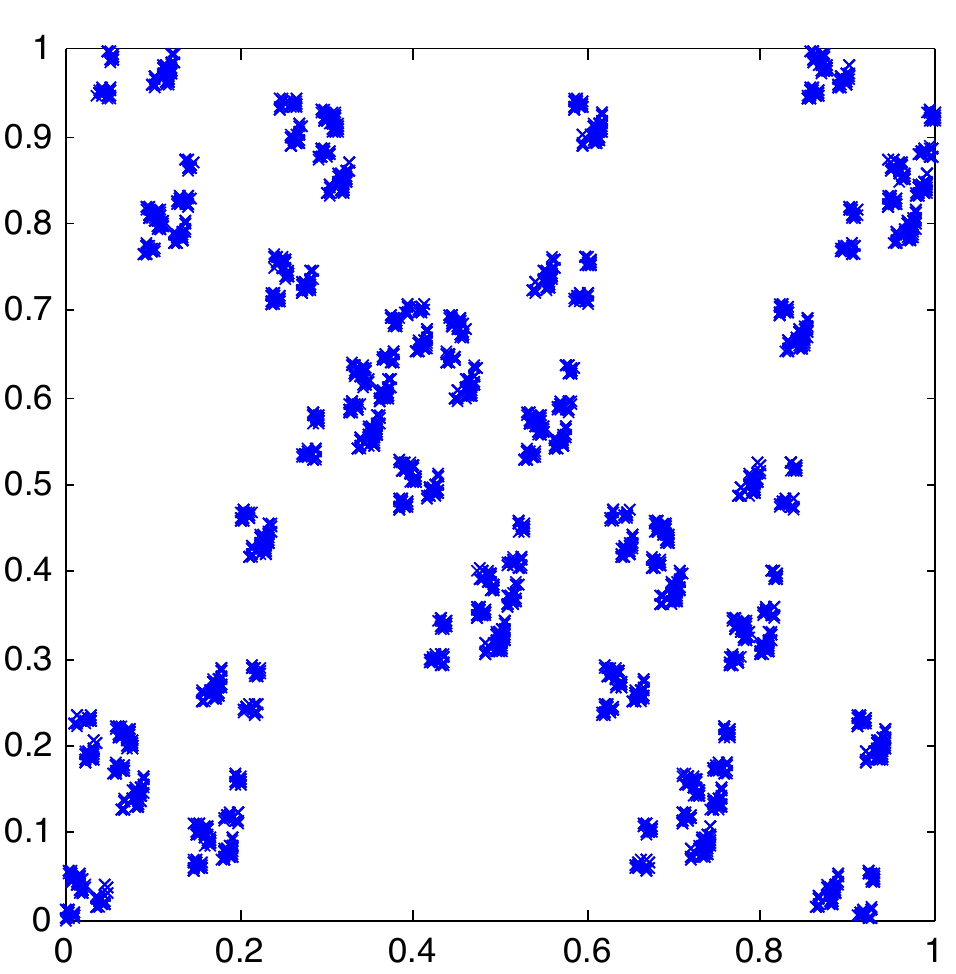}
\caption{\small{$\{\xi_{1,1}^{n},\xi_{4,1}^{n} \}_{n \ge 1}$}}
\end{minipage}%
\hspace{8mm}%
\begin{minipage}[c]{.45\textwidth}
\centering
\includegraphics[scale=0.43, keepaspectratio]{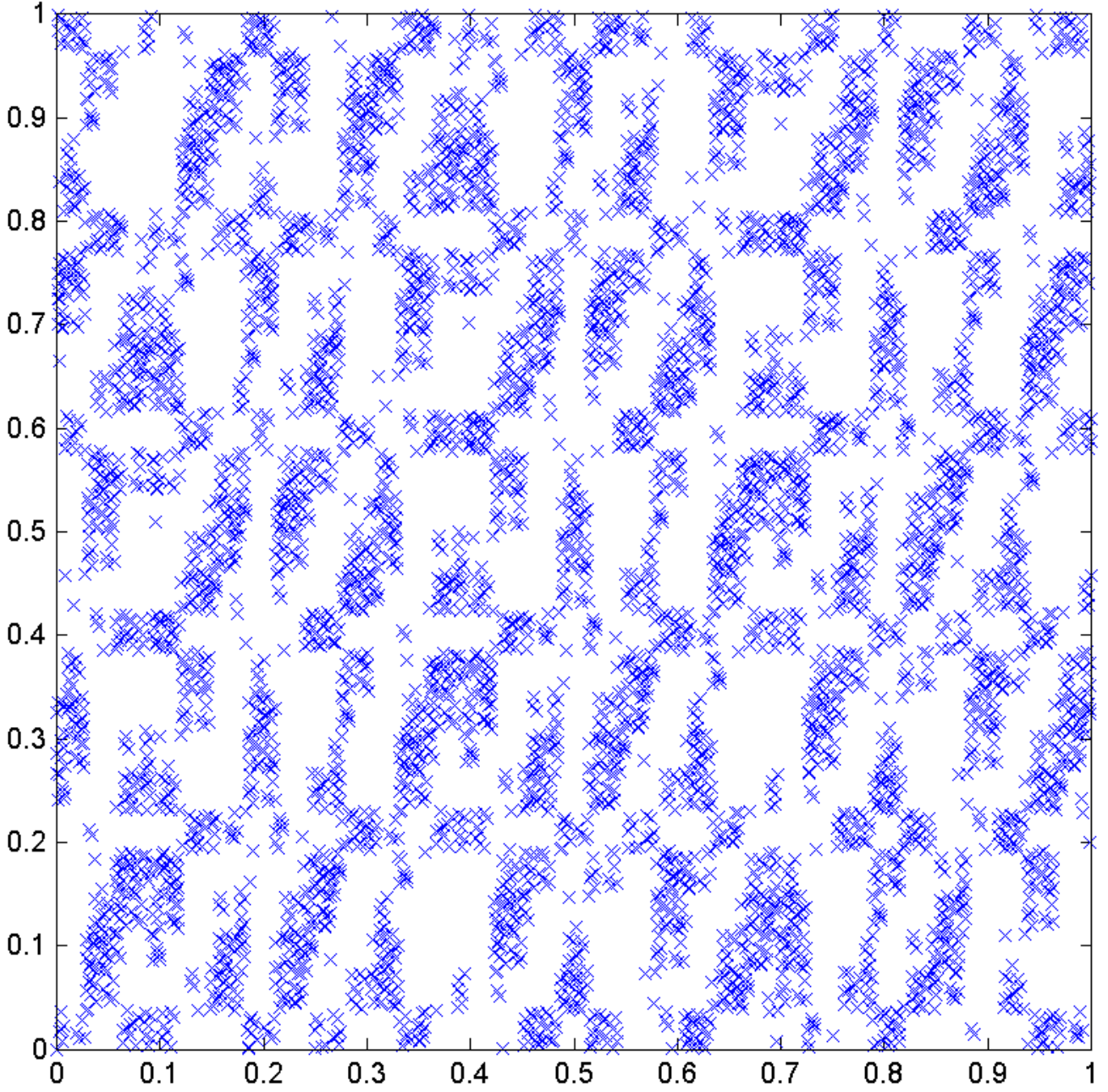}
\caption{\small{$\{\xi_{3,1}^{n},\xi_{5,1}^{n} \}_{n \ge 1}$}}
\end{minipage}
\end{figure}
\begin{figure}[h!]
\centering
\begin{minipage}[c]{.45\textwidth}
\centering
\includegraphics[scale=0.6, keepaspectratio]{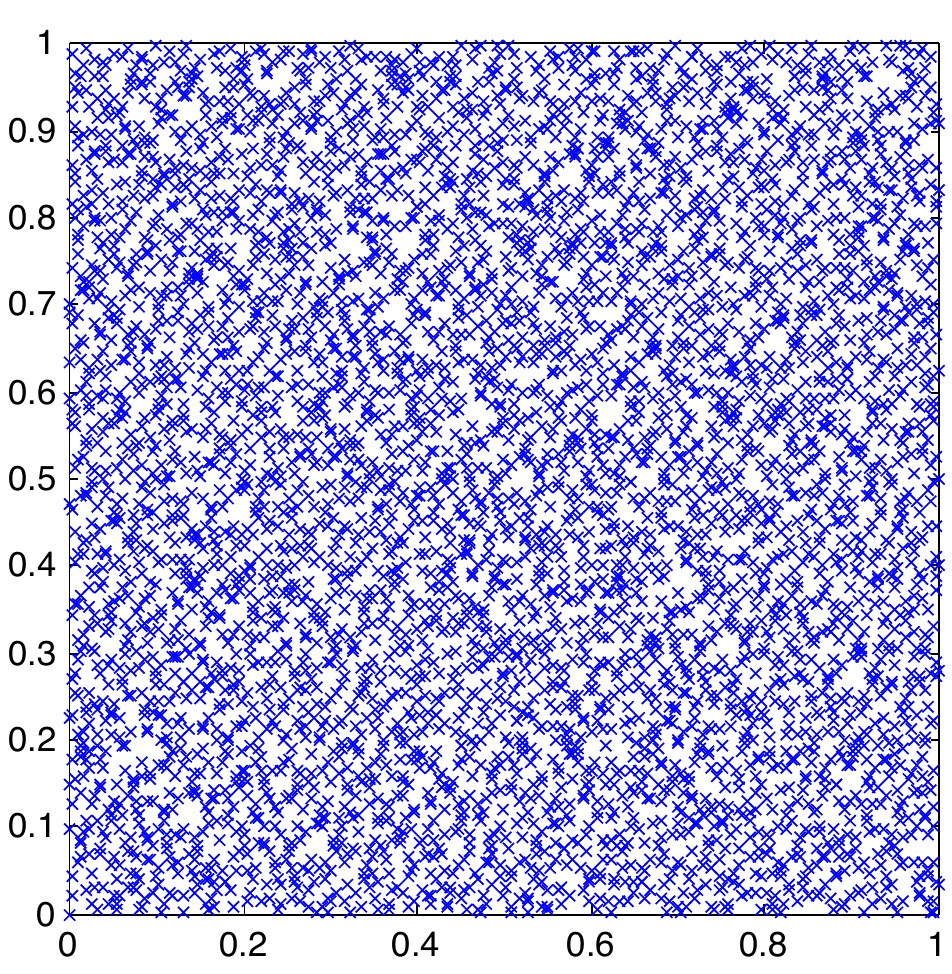}
\caption{\small{$\{\xi_{1,1}^{n},\xi_{5,1}^{n} \}_{n \ge 1}$}}
\end{minipage}%
\hspace{8mm}%
\begin{minipage}[c]{.45\textwidth}
\centering
\includegraphics[scale=0.6, keepaspectratio]{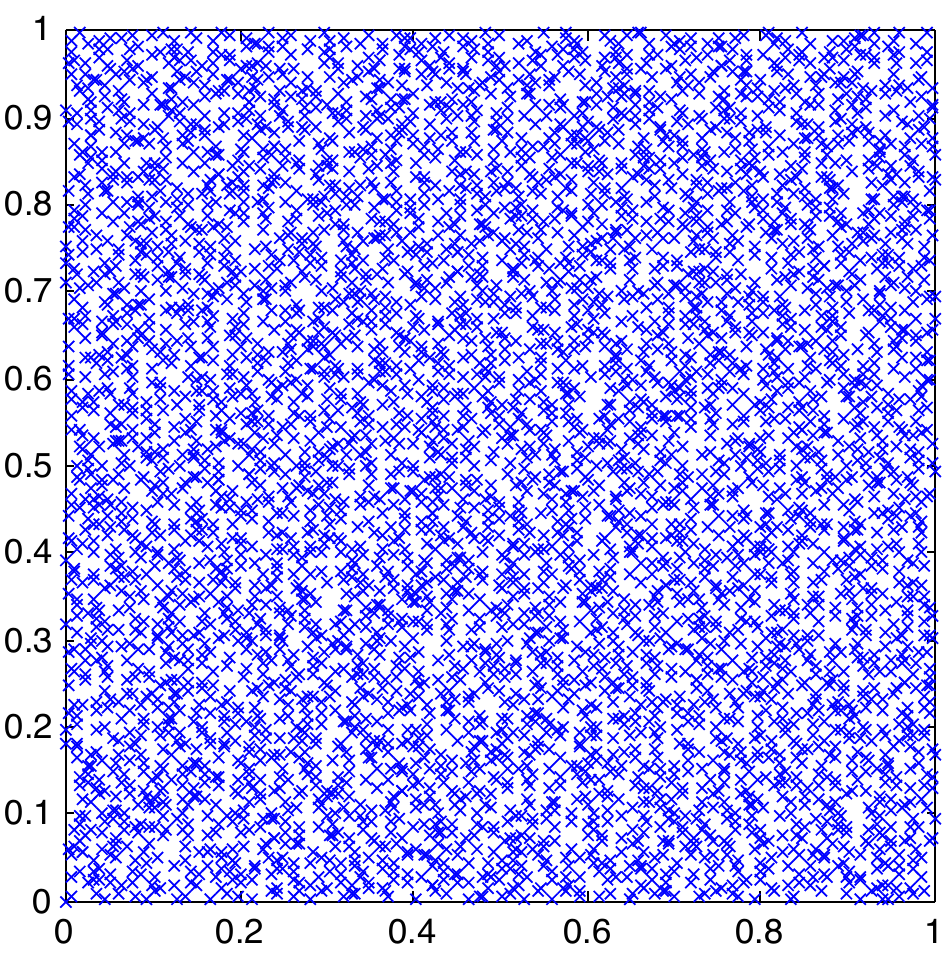}
\caption{\small{$\{\xi_{3,1}^{n},\xi_{4,1}^{n} \}_{n \ge 1}$}}
\end{minipage}
\end{figure}

Halton proved that  two van der Corput sequences  generate a uniformly distributed sequence in the square if and only if they have coprime bases. 

The situation is not clear at all if we use pairs of {\it $LS$-sequences of points}, as it can be seen in Figures 1 and 2, where numerical experiments suggest that there is no uniform distribution, while Figures 3 and 4 suggest that for those pairs of {\it $LS$-sequences of points} we do have uniform distribution.

  The following proposition may give an explanation for the ``resonance"  which appears in  Figure 1.
   
\begin{proposizione}\label{7} 
If for each $n\in\mathbb{N}$ we denote by $t_{n}$ the total number of intervals of $\rho_{1,1}^{n}$ and by $t_n^{'}$ the total number of intervals of the partition $\rho_{4,1}^{n}$, then we have that
$$t^{'}_n=t_{3n}\qquad \textrm{for\ all}\ n \in \mathbb{N}\ .$$
\end{proposizione}
\begin{proof} We use induction on $n \ge 1$.
If $n=1$ we have  $t^{'}_1=5=t_{3}$. 

Assume $t^{'}_{n-1}=t_{3(n-1)}$ holds. The relations $t_n=t_{n-1}+t_{n-2}$  and $ t^{'}_n=4t^{'}_{n-1}+t^{'}_{n-2} $ implies 
$t_{3n}= 4t_{3n-3}+t_{3n-6}$.
From the inductive assumption we have $4t_{3(n-1)}+t_{3(n-2)}=4t^{'}_{n-1}+t^{'}_{n-2}$, and therefore the proposition is proved. 
\end{proof}

{\bf Open problems}\\

1. When is $\{\xi_{L_1,S_1}^{n}, \xi_{L_2,S_2}^{n}\}_{n\geq 1}$ uniformly distributed?\\

2. What is its discrepancy?\\

3. What happens in higher dimension?

\bigskip
\noindent {\it Ingrid Carbone, Maria Rita Iac\`{o}*, Aljo\v{s}a Vol\v{c}i\v{c} }\\ 
University of Calabria, Department of Mathematics\\ 
Ponte P. Bucci Cubo 30B\\
87036 Arcavacata di Rende (Cosenza), Italy\\
E-mail: \texttt{i.carbone@unical.it}\\
E-mail: \texttt{volcic@unical.it}\\
E-mail: \texttt{iaco@mat.unical.it}

\bigskip
\noindent * 
  Research supported by the Doctoral Fellowship in Mathematics and Informatics of University of Calabria in cotutelle with Graz University of Technology, Institute of Mathematics A, Steyrergasse 30, 8010 Graz, Austria.

\end{document}